\documentclass[12pt,a4paper,reqno]{amsart}
\usepackage[english]{babel}
\usepackage[applemac]{inputenc}
\usepackage[T1]{fontenc}
\usepackage{palatino}
\usepackage{amsmath}
\usepackage{amssymb}
\usepackage{amsthm}
\usepackage{amsfonts}
\usepackage{graphicx}
\usepackage{mathtools}

\usepackage[colorlinks = true, citecolor = black]{hyperref}
\author{Tuomas Orponen}
\title{Quasisymmetric maps on Kakeya sets}
\address{University of Helsinki, Department of Mathematics and Statistics}
\subjclass[2010]{30L10 (Primary) 42B25, 28A78 (Secondary)}
\thanks{T.O. is supported by the Academy of Finland through the grant Restricted families of projections and connections to Kakeya type problems, grant number 274512.}
\email{tuomas.orponen@helsinki.fi}

\newcommand{\R}{\mathbb{R}}
\newcommand{\N}{\mathbb{N}}

\newcommand{\Z}{\mathbb{Z}}

\newcommand{\calH}{\mathcal{H}}

\newcommand{\Hd}{\dim_{\mathrm{H}}}

\newcommand{\diam}{\operatorname{diam}}

\numberwithin{equation}{section}

\theoremstyle{plain}
\newtheorem{thm}[equation]{Theorem}

\newtheorem{cor}[equation]{Corollary}
\newtheorem{proposition}[equation]{Proposition}

\theoremstyle{definition}

\newtheorem{definition}[equation]{Definition}

\theoremstyle{remark}

\begin{document}

\begin{abstract} I show that $L^{p}-L^{q}$ estimates for the Kakeya maximal function yield lower bounds for the conformal dimension of Kakeya sets, and upper bounds for how much quasisymmetries can increase the Hausdorff dimension of line segments inside Kakeya sets. 

Combining the known $L^{p}-L^{q}$ estimates of Wolff and Katz-Tao with the main result of the paper, the conformal dimension of Kakeya sets in $\R^{n}$ is at least $\max\{(n + 2)/2,(4n + 3)/7\}$. Moreover, if $f$ is a quasisymmetry from a Kakeya set $K \subset \R^{n}$ onto any at most $n$-dimensional metric space, the $f$-image of a.e. line segment inside $K$ has dimension at most $\min\{2n/(n + 2),7n/(4n + 3)\}$. The Kakeya maximal function conjecture implies that the bounds can be improved to $n$ and $1$, respectively. 
\end{abstract}

\maketitle

\section{Introduction}

A homeomorphism $g$ between two metric spaces $(X,d)$ and $(Y,d')$ is \emph{quasisymmetric}, if there exists an increasing homeomorphism $\eta \colon [0,\infty) \to [0,\infty)$ such that
\begin{displaymath} \frac{d'(g(x),g(y))}{d'(g(x),g(z))} \leq \eta\left(\frac{d(x,y)}{d(x,z)}\right) \end{displaymath}
for all triples of points $x,y,z \in X$ with $x \neq z$. A popular question in the study of these mappings is the following: given a specific metric space $(X,d)$ with Hausdorff dimension $\Hd X$, how much smaller can $\Hd Y = \Hd g(X)$ be, when $g \colon (X,d) \to (Y,d')$ is an arbitrary quasisymmetric homeomorphism? The infimum of these numbers $\Hd g(X)$ is known as the \emph{conformal dimension} of $X$. Quasisymmetric homeomorphisms are generally not bilipschitz-continuous, and indeed the dimension of $X$ alone warrants no non-trivial bounds for its conformal dimension. So, such bounds, if any, necessarily depend in a more delicate fashion on the structural properties of $X$. 

What kind of structural properties are relevant? A well-known principle is that strong connectivity of $X$ implies high conformal dimension. To make such a statement precise, one needs a way to quantify "strong connectivity", and this is often done via the concept of \emph{$p$-modulus}. Assume that $X$ supports a Borel measure $\mu$, and let $\Gamma$ be a family of curves in $X$. The $p$-modulus of $\Gamma$ with respect to $\mu$ is, by definition, the infimum of the numbers
\begin{displaymath} \int \rho^{p}(x) \ d\mu x, \end{displaymath}
where $\rho \colon X \to \R$ is any Borel function with the property that
\begin{displaymath} \int_{\gamma} \rho(x) \, d\calH^{1}(x) \geq 1 \end{displaymath}
for all the curves $\gamma \in \Gamma$. In short, the idea is that non-vanishing $p$-modulus for some $p > 0$ translates to connectivity within the support of $\mu$. For instance, observe that if $\calH^{1}(\gamma) \geq 1$ for all $\gamma \in \Gamma$, then a viable choice for $\rho$ is the characteristic function of the union $U_{\Gamma} := \bigcup_{\gamma \in \Gamma} \gamma$ (at least if $U_{\Gamma}$ happens to be Borel), so positive $p$-modulus for any $p > 0$ guarantees that $\mu(U_{\Gamma}) > 0$.

The following result of J. Tyson \cite{Ty} indicates that measuring "strong connectivity" in terms of positive modulus has merit:
\begin{thm}\label{tyson} Assume that $(X,d)$ is a compact $s$-Ahlfors-David regular metric space with $s > 1$. If $X$ contains a family of rectifiable curves with non-vanishing $s$-modulus, then the conformal dimension of $X$ is $s$.
\end{thm} 
A simple proof for Tyson's theorem can also be found in Heinonen's book \cite{He}, Theorem 15.10. Further validation for the concept of modulus is given by a result of Keith and Laakso \cite{KL}, which can be seen as a partial converse to Theorem \ref{tyson}: under the same assumptions, if the conformal dimension of $X$ is $s$, then some weak tangent of $X$ has non-vanishing $s$-modulus (in fact, even non-vanishing $1$-modulus). 

So, without doubt, estimating the modulus of curve families is a useful, and natural, method for studying conformal dimension. There are a few drawbacks, however. First, the very definition of modulus assumes an \emph{a priori} measure $\mu$ to inhabit $X$, and most (to the best of my knowledge) available results moreover require $\mu$ to be somewhat regular, doubling at least. Second, even if such a measure is given, and it happens to be $s$-Ahfors-David regular (as in Theorem \ref{tyson}), non-vanishing $s$-modulus for some curve family contained in $X$ is a rather strong hypothesis. 

A natural class of examples to illustrate both caveats simultaneously is given by \emph{Kakeya sets} in $\R^{n}$. The most commonly used definition says that $K \subset \R^{n}$ is a Kakeya set, if it contains a unit line-segment of the form $L_{e}(x) = \{x + te : 0 \leq t \leq 1\}$ for all unit vectors $e \in S^{n - 1}$. For technical purposes, I prefer a slightly relaxed version:
\begin{definition}\label{kakeyaSets} A set $K \subset \R^{n}$, $n \geq 2$, is called a Kakeya set, if there exists a set $S \subset S^{n - 1}$ of positive $\calH^{n - 1}$-measure such that $K$ contains a unit line segment of the form $L_{e}(x)$, as above, for every $e \in S$. 
\end{definition}
Kakeya sets are of great interest in the field of Euclidean harmonic analysis, and determining their Hausdorff dimension is a major open problem. The \emph{Kakeya conjecture} states that the Hausdorff dimension of Kakeya sets in $\R^{n}$ should always be $n$, but this is only known in the plane. An excellent in-depth introduction to the world of Kakeya sets can be found in P. Mattila's new book \cite{Ma}. For the immediate purposes of this introduction, however, it suffices to know two things: First, Kakeya sets contain a natural, rich family of curves, namely the unit line segments $L_{e}(x)$, $e \in S$. Second, in all dimensions, there exist Kakeya sets of measure zero; the first construction of these rather counter-intuitive sets is due to Besicovitch \cite{Be} from as early as 1919. 

Now, suppose that one would like to estimate the conformal dimension of a Kakeya set $K$ by applying the modulus method to the natural family of curves $\Gamma := \{L_{e}(x)\}_{e \in S}$. Several problems appear. What is a natural measure supported on $K$? And even if, hypothetically, $K$ supports a regular $s$-dimensional measure $\mu$, is there hope that the family $\Gamma$ has positive $s$-modulus with respect to $\mu$? The prospects start to look rather depressing, when one realises that even the answer to the second question is negative: consider the Kakeya set $K$ formed by the unit line segments centred at the origin. Then $K$ is the closed ball $B(0,1/2)$, so the existence of a natural measure is not in doubt. But, in this case, it is well-known that the family $\Gamma$ has vanishing $n$-modulus, see Corollary 7.20 in \cite{He}. Of course, one could try to argue that this particular $K$ contains other curve families with non-vanishing $n$-modulus, but this line of thought has little hope when pitted against Besicovitch's example of a Kakeya set with zero volume. 

Here is the first main result of this paper:
\begin{thm}\label{mainDim} The conformal dimension of Kakeya sets in $\R^{n}$ is at least 
\begin{displaymath} \max\left\{\frac{n + 2}{2},\frac{4n + 3}{7} \right\}. \end{displaymath}
\end{thm}
The bounds coincide for $n = 8$, and for $n \in \{3,4\}$, the bound equals the best known Hausdorff dimension estimate for Kakeya sets. After lamenting on the problems of using the modulus method in connection with Kakeya sets, it may seem strange to say that the proof of Theorem \ref{mainDim} is based on the modulus method. This is the case, however: the argument owes a great deal to that used for Theorem \ref{tyson}, as presented in Heinonen's book \cite{He}. The main new innovation, here, is the observation that a suitable $L^{p}-L^{p}$ inequality for the \emph{Kakeya maximal function}, defined shortly, can be inserted in the proof to replace the assumption of positive modulus. No \emph{a priori} measure on $K$ is needed, because these $L^{p}-L^{p}$ inequalities -- borrowed from the literature, see below -- are statements involving only Lebesgue measure in the ambient space. If a philosophical statement is permitted, one might say that the content of these $L^{p}-L^{p}$ inequalities is precisely that the modulus "vanishes quite slowly", and this is already sufficient to find a lower bound for conformal dimension.

The Kakeya maximal function $M_{\delta}$, with parameter $\delta > 0$, acts on locally integrable functions on $\R^{n}$ and produces functions defined on $S^{n - 1}$ through the equation
\begin{displaymath} (M_{\delta}f)(e) := \sup_{x \in \R^{n}} \frac{1}{|T^{\delta}_{e}(x)|} \int_{T_{e}^{\delta}(x)} |f(y)| \, dy. \end{displaymath}
Here $T_{e}^{\delta}(x)$ is a $\delta$-tube around the unit line segment $L_{e}(x) = \{x + te : 0 \leq t \leq 1\}$. It is well-known that $L^{p}-L^{q}$ estimates for $M_{\delta}$ guarantee lower bounds for the Hausdorff dimension of Kakeya sets, see Chapter 22 in Mattila's book \cite{Ma} for instance. Consequently, such estimates have been vigorously sought after for the past 20-30 years, at least. The main open problem in this setting is to verify the \emph{Kakeya maximal function conjecture}, which states that the inequality
\begin{equation}\label{kakeyaMaximal} \|M_{\delta}f\|_{L^{n}(S^{n - 1})} \leq C(\epsilon) \delta^{-\epsilon} \|f\|_{L^{n}(\R^{n})} \end{equation}
should hold for all $\epsilon > 0$ (with the constant $C(\epsilon)$ naturally blowing up as $\epsilon \to 0$; the constant also depends on $n$, but this number will be regarded as fixed and absolute throughout the paper). It is known that the Kakeya maximal function conjecture implies the Kakeya conjecture. This implication is not hard to prove as a stand-alone result, see Section 22.2 in \cite{Ma}, but it also follows by choosing $r = s = n$ and $u = 1$ in the second main theorem of the paper:
\begin{thm}\label{main} Let $u \geq 1$, $r > 0$ and $1 < s \leq n$. Assume that the Kakeya maximal function $M_{\delta}$ satisfies the estimate
\begin{equation}\label{aPriori} \|M_{\delta}f\|_{L^{r}(S^{n - 1})} \leq C(\epsilon)\delta^{-(n - s)/s - \epsilon}\|f\|_{L^{s}(\R^{n})} \end{equation}
for arbitrarily small $\epsilon > 0$, and with constants independent of $\delta > 0$. Then
\begin{displaymath} \Hd g(K) \geq su, \end{displaymath}
whenever $K \subset \R^{n}$ is a Kakeya set, and $g \colon K \to (X,d)$ is a quasisymmetric homeomorphism onto an arbitrary metric space $(X,d)$ with the property that 
\begin{equation}\label{lineDistortion} \calH^{n - 1}(\{e \in S^{n - 1} : \calH^{u}(g(L_{e}(x))) > 0\}) > 0. \end{equation}  \end{thm}

Here, and below $\calH^{u}$ stands for $u$-dimensional Hausdorff measure. Note that \eqref{lineDistortion} is an empty assumption for $u = 1$, since $\calH^{1}(g(L_{e}(x))) > 0$ for all $e \in S$. To deduce Theorem \ref{mainDim} from Theorem \ref{main}, it suffices to choose $u = 1$, and apply known bounds for the Kakeya maximal function, which happen to be of the form \eqref{aPriori}. Using the now famous "hairbrush construction", T. Wolff \cite{Wo} proved in 1995 that \eqref{aPriori} holds with 
\begin{displaymath} s = \frac{n + 2}{2} \quad \text{and} \quad r = (n - 1)s'= \frac{(n - 1)(n + 2)}{n}, \end{displaymath}
where $s' = s/(s - 1)$ is the dual exponent of $s$. Later, in 2002, N. Katz and T. Tao \cite{KT} employed methods from additive combinatorics to prove \eqref{aPriori} with 
\begin{displaymath} s = \frac{4n + 3}{7} \quad \text{and} \quad r = (n - 1)s' = \frac{4n + 3}{4}. \end{displaymath}
This numerology yields Theorem \ref{mainDim} immediately. Further, observing that the Kakeya maximal function conjecture is precisely \eqref{aPriori} with $r = s = n$, one obtains the following corollary:

\begin{cor} The Kakeya maximal function conjecture implies that the conformal dimension of Kakeya sets in $\R^{n}$ is $n$.
\end{cor} 

The best known lower bound for the Hausdorff dimension of Kakeya sets in dimensions $n \geq 5$ is $(2 - \sqrt{2})(n - 4) + 3$, as shown by Katz and Tao \cite{KT}. I do not know if the same bound is valid for conformal dimension.

What is the motivation for the hypothesis \eqref{lineDistortion} for $u > 1$? So far the introduction has concentrated on explaining the problem of bounding conformal dimension from below, but in fact \eqref{lineDistortion} has more to do with another (closely related) line of research: proving that quasiconformal mappings do \textbf{not} increase the dimension of "most" curves. The seminal result in this direction is due to F. Gehring \cite{Ge,Ge2}, who showed that a quasiconformal self-map of $\R^{n}$ is absolutely continuous on almost all lines in $\R^{n}$: in particular, the dimension of the images of those lines is $1$. Since the 60's, Gehring's result has been improved and generalised by many: most developments are collected in the excellent introduction to the paper of Z. Balogh, P. Koskela and S. Rogovin \cite{BKR}, which also (to the best of my knowledge) contains the state-of-the-art result:
\begin{thm}[Balogh-Koskela-Rogovin, 2007]\label{BKR} Assume that $X,Y$ are locally $Q$-Ahlfors-David regular spaces with $Q > 1$. Suppose that $X$ is proper, and $f \colon X \to Y$ is a homeomorphism with uniformly bounded (lower) distortion almost everywhere, and finite (lower) distortion outside an exceptional set with $\sigma$-finite $(Q - 1)$-dimensional measure. Then $f$ is absolutely continuous on $1$-modulus almost every curve. 
\end{thm}

For the definitions, see \cite{BKR}. For the purposes of this introduction, it suffices to know that the assumptions on $f$ in Theorem \ref{BKR} are significantly weaker than quasisymmetry, but still of the same general flavour.

Now, suppose that one is interested in the behaviour of an arbitrary quasisymmetric homeomorphism on generic lines in a Kakeya set $K \subset \R^{n}$. As before, due to the lack of a natural regular measure on $K$, or the target space $(X,d)$, Theorem \ref{BKR} sheds no light on the situation. Also, absolute continuity is too much to ask for in this generality: replacing the Euclidean metric on $K$ by a slightly snow-flaked version, say $d$, one can arrange so that $\operatorname{Id} \colon K \to (K,d)$ is a quasisymmetric homeomorphism with $\Hd K = \Hd (K,d)$, but $\operatorname{Id} \colon K \to (K,d)$ is not absolutely continuous on any line in $K$. To be slightly more precise, assume that $K$ is contained in a ball $B$ of (Euclidean) diameter $5$. Inside $B$ (and on $K$ in particular), consider the metric $d$ defined by
\begin{displaymath} d(x,y) = |x - y| \cdot \ln \frac{100}{|x - y|}. \end{displaymath}
This is indeed a metric in $B$, since $\psi(t) = t \cdot \ln (100/t)$ is strictly increasing and concave on the interval $[0,10]$, and $\psi(0) = 0$. However, the mapping $\operatorname{Id} \colon K \to (K,d)$ is not absolute continuous on line segments, because line segments do not have $\sigma$-finite $\calH^{1}_{d}$-measure. So, in short: absolute continuity is such a delicate property that assuming some regularity from the domain and target, as in Theorem \ref{BKR}, seems practically necessary. 

The property of \emph{not increasing the dimension of generic lines (too much)} is weaker than \emph{being absolutely continuous on generic lines}, and hence, possibly, more easily achieved without so stringent regularity hypotheses. Indeed, in combination with the previously mentioned Kakeya maximal function estimates of Wolff and Katz-Tao, Theorem \ref{main} has the following corollary:
\begin{cor}\label{mainCor2} Assume that $K \subset \R^{n}$ is a Kakeya set, $(X,d)$ is a metric space with $\Hd X \leq n$, and $g \colon K \to (X,d)$ is a quasisymmetric homeomorphism. Then
\begin{displaymath} \Hd g(L_{e}(x)) \leq \min \left\{\frac{2n}{n + 2},\frac{7n}{4n + 3} \right\} \end{displaymath}
for $\calH^{n - 1}$-almost all line segments $L_{e}(x) \subset K$. \end{cor}

\begin{proof} If the corollary fails, then \eqref{lineDistortion} holds for some $u > \min\{2n/(n + 2),7n/(4n + 3)\}$. If for, instance, $u > 2n/(n + 2)$ one relies on Wolff's estimate, which says that \eqref{aPriori} holds with $s = (n + 2)/2$. Then, by Theorem \ref{main}, 
\begin{displaymath} \Hd X = \Hd g(K) \geq su  > \frac{n + 2}{2} \cdot \frac{2n}{n + 2} = n, \end{displaymath} 
which is a contradiction. If $u > 7n/(4n + 3)$, use the Katz-Tao bound instead. \end{proof}
As before, the Kakeya maximal function conjecture would give optimal results:
\begin{cor} With the same notation as above, The Kakeya maximal function conjecture implies that $\dim g(L_{e}(x)) = 1$ for almost all line segments $L_{e}(x) \subset K$. \end{cor} 

\section{Proof of the main theorem}

This section contains the proof of the technical main result, Theorem \ref{main}. The argument only uses very basic facts about quasisymmetric maps, most of which are easiest to verify on-the-go. The only fact worth citing from the literature at this point is the following inequality, see Proposition 10.8 in Heinonen's book \cite{He}:
\begin{proposition}\label{diamComparison} Assume that $g \colon (X,d) \to (Y,d')$ is a quasysymmetric embedding with associated homeomorphism $\eta \colon [0,\infty) \to [0,\infty)$. Then, for all sets $A \subset A' \subset X$ with $\diam A' < \infty$, it holds that $\diam g(A') < \infty$, and
\begin{displaymath} \frac{1}{2\eta\left(\frac{\diam A'}{\diam A}\right)} \leq \frac{\diam g(A)}{\diam g(A')} \leq \eta\left(\frac{2\diam A}{\diam A'} \right). \end{displaymath}
\end{proposition}
A word on notation: $B = B(x,r)$ always stands for a closed ball centred at $x$, with radius $r > 0$. The $\lambda$-enlargement of $B$ is denoted by $\lambda B := B(x,\lambda r)$. The inequality $A \lesssim_{p} B$ means that there is a constant $C = C(p) \geq 1$ such that $A \leq CB$. The two-sided inequality $A \lesssim_{p} B \lesssim_{q} A$ is shortened to $A \sim_{p,q} B$.

It remains to prove Theorem \ref{main}:
\begin{proof}[Proof of Theorem \ref{main}] Fix a quasisymmetric homeomorphism $g \colon K \to (X,d)$ with associated homeomorphism $\eta \colon [0,\infty) \to [0,\infty)$. Since every Kakeya set contains a bounded Kakeya set in the sense of Definition \ref{kakeyaSets}, one may assume that $K$ is bounded. Consequently, also $g(K)$ is bounded by Proposition \ref{diamComparison}. Further, one may assume that 
\begin{displaymath} \calH^{u}_{\infty}(g(L_{e})) \geq \tau > 0 \end{displaymath}
for all $e \in S$, where $\calH^{u}_{\infty}$ stands for $u$-dimensional Hausdorff content, and $S \subset S^{n - 1}$ and $L_{e} := L_{e}(x)$ are as in Definition \ref{kakeyaSets}: otherwise a Kakeya subset of $K$ has this property. Finally, one may assume that every point in $K$ is contained in one of the line segments in $S$, that is, $K = \bigcup_{e \in S} L_{e}$. 

Now, to reach a contradiction, suppose that $\Hd g(K) < t < s u$. Then, fix $\rho \in (0,1)$ to be specified later, and choose a sequence of balls $B_{1}^{X},B_{2}^{X},\ldots$ with diameters $d_{1}^{X},d_{2}^{X}\ldots$ such that
\begin{equation}\label{form8} g(K) \subset \bigcup_{j \in \N} B_{j}^{X}, \end{equation}
and
\begin{equation}\label{form7} \sum_{j \in \N} (d_{j}^{X})^{t} < \rho. \end{equation}
For some constant $Q \geq 1$ depending only on the quasisymmetry of $g$, one may choose balls $B_{1},B_{2},\ldots \subset \R^{n}$, centred at points in $K$ and with diameters $d_{1},d_{2},\ldots$ such that
\begin{equation}\label{form1} K \cap B_{j} \subset g^{-1}(B_{j}^{X}) \subset K \cap QB_{j}. \end{equation}
Since $\diam g^{-1}(B_{j}^{X}) \leq \diam K < \infty$, one may assume that $d_{j} \lesssim_{\diam K} 1$. Also, since each point in $K$ lies on a unit line segment contained in $K$, and the balls $B_{j}$ are centred on $K$, one has 
\begin{equation}\label{form9} \diam (K \cap B_{j}) \gtrsim \min\{d_{j},1\} \gtrsim d_{j}. \end{equation}
The latter implicit constants depend on $\diam K$, but this (as well as $\diam g(K)$) will be regarded as absolute constants. One may also assume that the balls $B_{j}$ are disjoint. Indeed, by a standard Vitali-type covering lemma, one can pick a disjoint sub-collection $\{QB_{i_{j}}\} \subset \{QB_{j}\}$ such that the union of the balls $5QB_{i_{j}}$ cover the union of the balls $QB_{j}$. Then, the images $g(K \cap 5QB_{i_{j}})$ cover $g(K)$, and $\diam g(K \cap 5QB_{i_{j}}) \lesssim d_{i_{j}}^{X}$, where the implicit constants only depend on the the quasisymmetry of $g$ and the diameter of $K$ (this follows from Proposition \ref{diamComparison} and \eqref{form9}).  One may now replace the balls $B_{i_{j}}^{X}$ by some balls $A_{i_{j}}^{X}$ containing $g(K \cap 5QB_{i_{j}})$, and with $\diam A_{i_{j}}^{X} \sim \diam g(K \cap 5QB_{i_{j}}) \lesssim d_{i_{j}}^{X}$. Then \eqref{form8} and \eqref{form7} hold for the balls $A_{i_{j}}^{X}$ and their diameters, and the analogue of \eqref{form1} is
\begin{displaymath} K \cap QB_{i_{j}} \subset g^{-1}(A_{i_{j}}^{X}) \subset K \cap Q'B_{i_{j}}, \end{displaymath}
where again the constant $Q' \geq 5Q$ only depends on the quasisymmetry of $g$. Now the balls $QB_{i_{j}}$ on the left hand side are disjoint, as desired, and the proof may continue.

Consider the function $f \colon \R^{n} \to \R$, defined by
\begin{displaymath} f := \sum_{j \in \N} \frac{(d_{j}^{X})^{u}}{d_{j}} \chi_{2QB_{j}}. \end{displaymath}
For any unit line segment $L_{e} \subset K$, $e \in S$, one has
\begin{displaymath} \int_{L_{e}} f(x) \, d\calH^{1}(x) = \sum_{j \in \N} \frac{(d_{j}^{X})^{u}}{d_{j}} \calH^{1}(L_{e} \cap 2QB_{j}) \geq \sum_{j : B_{j}^{X} \cap g(L_{e}) \neq \emptyset} \frac{(d_{j}^{X})^{u}}{d_{j}} \calH^{1}(L_{e} \cap 2QB_{j}). \end{displaymath}
If $B_{j}^{X} \cap g(L_{e}) \neq \emptyset$, then $L_{e} \cap QB_{j} \neq \emptyset$ by the right hand side inclusion of \eqref{form1}, and it follows that $\calH^{1}(L_{e} \cap 2QB_{j}) \gtrsim \min\{d_{j},1\} \gtrsim d_{j}$. Consequently,
\begin{equation}\label{form2} \int_{L_{e}} f(x) \, d\calH^{1}(x) \gtrsim \sum_{j : B_{j}^{X} \cap g(L_{e}) \neq \emptyset} (d_{j}^{X})^{u} \geq \calH^{u}_{\infty}(g(L_{e})) \geq \tau, \end{equation}
since the balls $B_{j}^{X}$ with $B_{j}^{X} \cap g(L_{e}) \neq \emptyset$ cover $g(L_{e})$.

Equation \eqref{form2} is next used to pigeonhole a suitable scale. For $k \in \Z$, define the index set 
\begin{displaymath} N_{k} := \{j \in \N : p^{-k - 1} \leq d_{j} < p^{-k} \text{ and } 2QB_{j} \cap L_{e} \neq \emptyset \text{ for some } e \in S\}, \end{displaymath}
where $p \in \N$ is an integer with the property that $\eta(1/p) \leq 1/2$, and observe that $N_{k}$ is empty for $k < k_{0}$, where $k_{0}$ depends only on $\max d_{j} \lesssim 1$. It follows from \eqref{form2} that for each $e \in S$, there is an index $k \geq k_{0}$ with
\begin{displaymath} \int_{L_{e}}  \sum_{j \in N_{k}} \frac{(d_{j}^{X})^{u}}{d_{j}} \chi_{2QB_{j}} \, d\calH^{1} \gtrsim \frac{1}{|k|^{2} + 1}. \end{displaymath}
Further, since this holds for all $e \in S$, the Borel-Cantelli lemma implies that one can choose the index $k \geq k_{0}$ so that
\begin{equation}\label{form3} \calH^{n - 1}\left(\left\{ e \in S : \int_{L_{e}}  \sum_{j \in N_{k}} \frac{(d_{j}^{X})^{u}}{d_{j}} \chi_{2QB_{j}} \, d\calH^{1} \gtrsim \frac{1}{|k|^{2} + 1} \right\} \right) \gtrsim \frac{\calH^{n - 1}(S)}{|k|^{2} + 1}. \end{equation}
Write $\delta := p^{-k - 1}$ for this $k \geq k_{0}$, and consider $M_{\delta}f_{k}$, where
\begin{displaymath} f_{k} := \sum_{j \in N_{k}} \frac{(d_{j}^{X})^{u}}{d_{j}} \chi_{4QB_{j}}. \end{displaymath}  
Recall that $T_{e}^{\delta}$ is the closed $\delta$-neighbourhood of the line segment $L_{e}$. Since
\begin{displaymath} \frac{1}{\calH^{n}(T_{e}^{\delta})} \int_{T_{e}^{\delta}} \chi_{4QB_{j}} \,d\calH^{n} \gtrsim \int_{L_{e}} \chi_{2QB_{j}} \, d\calH^{1} \end{displaymath}
for any $B_{j}$ with $\diam B_{j} \geq \delta$, one has $M_{\delta}f_{k}(e) \gtrsim 1/(|k|^{2} + 1)$ for all $e$ in the subset of $S$ defined in \eqref{form3}. Consequently, for any $\epsilon > 0$,
\begin{align} \left(\frac{\calH^{n - 1}(S)}{(|k|^{2} + 1)^{1 + r}}\right)^{1/r} & \lesssim \left( \int_{S^{n - 1}} [M_{\delta}f_{k}]^{r} \, d\calH^{n - 1} \right)^{1/r}\notag\\
&\label{form4} \leq C(\epsilon) \delta^{-(n - s)/s - \epsilon} \left( \int_{\R^{n}} f_{k}^{s} \, d\calH^{n} \right)^{1/s}, \end{align}
using the main hypothesis on the $L^{s}-L^{r}$ boundedness of $M_{\delta}$. Since $s > 1$, and the balls $B_{j}$ are disjoint, one can further estimate as follows:
\begin{align} \int_{\R^{n}} f_{k}^{s} \, d\calH^{n} & = \int_{\R^{n}} \left[ \sum_{j \in N_{k}} \frac{(d_{j}^{X})^{u}}{d_{j}}\chi_{4QB_{j}} \right]^{s} \, d\calH^{n} \notag\\
& \lesssim_{Q,n,s} \int_{\R^{n}} \left[ \sum_{j \in N_{k}} \frac{(d_{j}^{X})^{u}}{d_{j}}\chi_{B_{j}} \right]^{s} \, d\calH^{n} \notag\\
&\label{form5} \sim \delta^{n - s} \sum_{j \in N_{k}} (d_{j}^{X})^{su} = \delta^{n - s}\sum_{j \in N_{k}} (d_{j}^{X})^{t} \cdot (d_{j}^{X})^{su - t}. \end{align}
The inequality on the second line is standard, using the doubling of the measure $\calH^{n}$, and the fact that $s > 1$, see Exercise 2.10 in Heinonen's book \cite{He}.

There are now two cases: either $\delta$ is small or large. Specifically, assume first that $2Qp\delta = 2Qp^{-k} \geq 1/(100p)$. Then also $k \sim_{p,Q} 1$, and a combination of \eqref{form4} and \eqref{form5} with $\epsilon = 1/2$, say, gives 
\begin{displaymath} \calH^{n - 1}(S) \lesssim_{n,p,Q} \sum_{j \in N_{k}} (d_{j}^{X})^{t} < \rho. \end{displaymath}
Now, if $\rho > 0$ is small enough (depending on $n,p$, $\diam K$ and $Q$), this is impossible, so one may assume that $2Qp^{-k} \leq 1/(100 p)$ in the sequel.

To take advantage of the small factor $(d_{j}^{X})^{s u - t}$ in \eqref{form5} (recall that $su - t > 0$), one wishes to bound $d_{j}^{X}$ from above by some small power of $\delta$ for $j \in N_{k}$, that is, whenever $d_{j} \sim \delta = p^{-k - 1}$. Roughly speaking, this is possible because quasisymmetries are locally Hölder-continuous (with index depending only on the homeomorphism $\eta$). One could cite existing results here, but a self-contained proof is so simple that I give the details. Fix $B = B_{j}$ with $j \in N_{k}$, so that in particular $2QB \cap L_{e} \neq \emptyset$ for some $e \in S$. Then, choose an increasing sequence of concentric balls 
\begin{displaymath} B = B^{1} \subset B^{2} \subset \ldots \subset B^{N}, \end{displaymath}
where $QB^{N} \supset K$ with $\diam QB^{N} \sim \diam(K) \sim 1$, and 
\begin{equation}\label{form6} \frac{2\diam (K \cap QB^{i})}{\diam (K \cap QB^{i + 1})} \leq \frac{1}{p}. \end{equation}
Such balls are easy to construct: assume that $B^{i}$ has already been chosen so that $B^{i} \supset B$ and $\diam 2QB^{i} \leq 1/(100p)$; note that this is true for $B^{1} = B$, because $\diam B = \diam B_{j} = d_{j} \in [p^{-k - 1},p^{-k})$, and $2Qp^{-k} \leq 1/(100p)$. Since $2QB^{i} \cap L_{e} \neq \emptyset$ for some segment $L_{e} \subset K$, one has 
\begin{displaymath} \diam (K \cap 100pQB^{i}) \geq \diam (L_{e} \cap 100pQB^{i}) \geq 2p\diam (QB^{i}) \geq 2p \diam(K \cap QB^{i}), \end{displaymath}
and hence $B^{i + 1} := 100pQB^{i}$ satisfies \eqref{form6}. Continuing like this, after $N \gtrsim_{p} k$ steps, one reaches a stage, where $\diam 2QB^{N - 1} \leq 1/(100p)$, but $\diam  2QB^{N} > 1/(100p)$. Then, one enlarges $B^{N}$ a bit so that $QB^{N} \supset K$, and the definition of the sequence $(B^{i})_{i = 1}^{N}$ is complete. 

Now, recall that $d_{j}^{X} = \diam B_{j}^{X}$, where $B_{j}^{X}$ is the ball in $X$ associated to $B = B_{j}$. Then, since $B_{j}^{X} \subset g(K \cap QB)$ by \eqref{form1}, one has
\begin{displaymath} \frac{d_{j}^{X}}{\diam g(K)} \leq \frac{\diam g(K \cap QB)}{\diam g(K \cap QB^{N})} \leq \prod_{i = 1}^{N - 1} \frac{\diam g(K \cap QB^{i})}{\diam g(K \cap QB^{i + 1})}. \end{displaymath}
To estimate the product on the right hand side, one applies the right hand side inequality of Proposition \ref{diamComparison}, namely
\begin{displaymath} \frac{\diam g(A)}{\diam g(A')} \leq \eta\left(\frac{2\diam A}{\diam A'}\right)\quad \text{for } A \subset A' \subset \R^{n}. \end{displaymath}
In particular, by the defining property \eqref{form6} of the sequence $(B_{i})_{i = 1}^{N}$, one has
\begin{displaymath} \prod_{i = 1}^{N - 1} \frac{\diam g(K \cap QB^{i})}{\diam g(K \cap QB^{i + 1})} \leq \prod_{i = 1}^{N} \eta\left(\frac{2\diam (K \cap QB^{i})}{\diam (K \cap QB^{i + 1})}\right) \leq \eta\left(\frac{1}{p}\right)^{N - 1} \leq \left(\frac{1}{2}\right)^{N - 1}, \end{displaymath} 
recalling the choice of $p \in \N$. Since $N \gtrsim_{p} k$, one infers that 
\begin{displaymath} d_{j}^{X} \lesssim_{p} (1/2)^{k} = p^{-\gamma k} \sim_{p} \delta^{\gamma} \end{displaymath}
with $\gamma = \log_{p} 2 > 0$ (which is indeed a constant depending only on the function $\eta$). Finally, a combination of this estimate with \eqref{form4} and \eqref{form5} gives
\begin{displaymath} \left(\frac{\calH^{n - 1}(S)}{(|k|^{2} + 1)^{r + 1}}\right)^{1/r} \lesssim_{p} C(\epsilon) \delta^{\gamma(su - t)/s - \epsilon} \left[\sum_{j \in \N} (d_{j}^{X})^{t}\right]^{1/s} \leq C(\epsilon) \delta^{\gamma(su - t)/s - \epsilon}\rho^{1/s}.  \end{displaymath} 
To complete the proof, one chooses $\epsilon > 0$ so small that $\gamma(su - t)/s - \epsilon \geq \gamma(su - t)/(2s)$. Then, recalling that $\delta = p^{-k - 1}$, hence $k \sim \log_{p} (1/\delta)$, one ends up with
\begin{displaymath} \calH^{n - 1}(S) \lesssim_{p,\gamma,s,t,u} \delta^{\gamma r(su - t)/(2s)} (\log_{p} (1/\delta))^{r + 1} \rho^{r/s}. \end{displaymath}
Since the factor $\delta^{\gamma r(su - t)/(2s)} (\log_{p} (1/\delta))^{r + 1}$ is uniformly bounded for the relevant (that is: small) values of $\delta$, choosing $\rho > 0$ small enough produces a contradiction. The proof is complete.  
\end{proof}

\end{document}